\documentclass{amsart}
\newtheorem{lemma}{Lemma}

\newtheorem{theorem}{Theorem}
\newtheorem{corollary}{Corollary}

\title{On a Generalization of the Frobenius Number}

\author{Alexander Brown \and Eleanor Dannenberg \and Jennifer Fox \and Joshua Hanna \and  Katherine Keck \and Alexander Moore \and Zachary Robbins \and Brandon Samples \and James Stankewicz}
\email{stankewicz@gmail.com}
\address{Department of Mathematics, University of Georgia, Athens, GA 30602}

\begin{document}

\begin{abstract}

We consider a generalization of the Frobenius Problem where the object of interest is the greatest integer which has exactly $j$ representations by a collection of positive relatively prime integers. We prove an analogue of a theorem of Brauer and Shockley and show how it can be used for computation.

\end{abstract}

\maketitle

The \emph{linear diophantine problem of Frobenius} has long been a celebrated problem in number theory. Most simply put, the problem is to find the \emph{Frobenius Number} of $k$ positive relatively prime integers $(a_1, \ldots, a_k)$, i.e., the greatest integer $M$ for which there is no way to express $M$ as the non-negative integral linear combination of the given $a_i$.

A generalization, which has drawn interest both from classical study of the Frobenius Problem (\cite[Problem A.2.6]{Alf}) and from the perspective of partition functions and integer points in polytopes (as in Beck and Robins \cite{1}), is to ask for the greatest integer $M$ which can be expressed in exactly $j$ different ways. We make this precise with the following definitions:

A \emph{representation} of $M$ by a $k$-tuple $(a_1, \ldots, a_k)$ of non-negative, relatively prime integers is a solution $(x_1, \ldots, x_k)\in \mathbb{Z}^k_{\ge 0}$ to the equation $M = \sum_{i=1}^k a_ix_i$.

We define the $j$-\emph{Frobenius Number} of a $k$-tuple $(a_1, \ldots, a_k)$ of relatively prime positive integers to be the greatest integer $M$ with exactly $j$ representations of $M$ by $(a_1, \ldots, a_k)$ if such a positive integer exists and zero otherwise. We refer to this quantity as $g_j(a_1, \ldots, a_k)$.

Finally we define $f_j(a_1, \ldots, a_k)$ exactly as we defined $g_j(a_1, \ldots, a_k)$, except that we consider only \emph{positive representations} $(x_1, \ldots, x_k)\in \mathbb{Z}^k_{> 0}$.

Note that the $0$-Frobenius of $(a_1, \ldots, a_k)$ is the classical Frobenius Number. The purpose of this paper is to show the following generalization of a result of Brauer and Shockley \cite{2} on the classical Frobenius Number.

\begin{theorem}\label{bigresult} If $d=\gcd(a_2, \ldots, a_k)$ and $j\ge 0$, then either

$$g_j(a_1, a_2, \ldots,a_k) = d\cdot g_j(a_1, \frac{a_2}{d}, \ldots, \frac{a_k}{d}) + (d-1)a_1$$

or $g_j(a_1,a_2, \ldots,a_k) = g_j(a_1, \frac{a_2}{d}, \ldots, \frac{a_k}{d}) = 0.$\end{theorem}

The research done in this note was completed as part of the undergraduate number theory VIGRE research seminar directed by Professor Dino Lorenzini and assisted by graduate students Brandon Samples and James Stankewicz at the University of Georgia in Fall 2008. Support for the seminar was provided by the Mathematics Department's NSF VIGRE grant.

\begin{lemma}\label{lemma}

If $f_j(a_1, \ldots, a_k)$ is nonzero, there exist integers $x_2, \ldots, x_k > 0$ such that $$f_j(a_1, \ldots, a_k) = \sum_{i=2}^k a_i x_i.$$

\end{lemma}

\begin{proof}Let $f_j := f_j(a_1, \ldots, a_k)$. By the definition of $f_j$, we can write $f_j = \displaystyle\sum_{i=1}^k a_i x_{i,\ell}$ with $x_{i,\ell} > 0$ for $1 \leq \ell \leq j$.  Since
$$f_j + a_1 = \displaystyle \sum_{i=1}^k a_i x_{i,\ell} + a_1= a_1(x_{1,\ell} + 1) + \displaystyle \sum_{i=2}^k a_i x_{i,\ell},$$
we obtain at least $j$ positive representations of $f_j + a_1$.  As $f_j$ is the largest number with exactly $j$ positive representations,  there must be at least $j+1$ distinct ways to represent $f_j + a_1$.  Specifically, we have $f_j + a_1 = \displaystyle \sum_{i=1}^k a_i x'_{i,\ell}$ with $x'_{i,\ell} > 0$ for all $1 \leq \ell \leq j+1$.  Subtract $a_1$ from both sides of these $j+1$ equations to obtain $f_j = (x'_{1,{\ell}} - 1)a_1 + \displaystyle \sum_{i=2}^k a_i x'_{i,\ell}$.  Evidently, there exists some $\ell_0 \in [1, j+1]$ for which $x'_{1,{\ell_0}} - 1 = 0$ because $f_j$ cannot have $j+1$ positive representations.  Therefore, $f_j(a_1, \ldots, a_k) = \displaystyle \sum_{i=2}^k a_i x'_{i,{\ell_0}}.$\end{proof}

\begin{theorem}\label{theorem}

If $\mathrm{gcd}(a_2,\ldots,a_k)=d$, then $$f_j(a_1,a_2,\ldots,a_k)=d\cdot f_j(a_1,\dfrac{a_2}{d},\ldots,\dfrac{a_k}{d}).$$ 

\end{theorem}

\begin{proof}Let $a_i=da'_i$ for $i=2,\ldots,k$ and $N=f_j(a_1,\ldots,a_k)$.\\
Assuming $N>0$, we know by Lemma \ref{lemma} that $$N=\sum_{i=2}^k a_i x_i=d\sum_{i=2}^k a'_i x_i$$ with $x_i>0$. Let $N'=\sum_{i=2}^k a'_i x_i$. We want to show that $N'=f_j(a_1,a'_2,\ldots,a'_k)$ and will do this in three steps.\\ \\
\textbf{Step 1:}  First, we know that $N'$ does not have $j+1$ or more positive representations by   $a_1,a'_2,\ldots,a'_k$.  If $N'$ could be so represented, then for $1\leq l\leq j+1$ we would have
$$N'=a_1y_{{1,\ell}} + \sum_{i=2}^k a'_i y_{{i,\ell}}.$$
Multiplying this equation by $d$ immediately produces too many representations of $N$ and thus a contradiction.\\ \\
\textbf{Step 2:} Next, we know that
$$f_j(a_1,\ldots, a_k) =N=a_1x_{{1,\ell}} + \sum_{i=2}^k a_i x_{{i,\ell}}$$ 
for $1\leq l\leq j$ and $x_i>0$, so
$$\dfrac{N}{d}=\dfrac{a_1x_{{1,\ell}}}{d} + \sum_{i=2}^k \dfrac{a_i x_{{i,\ell}}}{d}.$$
Since $d|N$ and $d|a_i$ for $i\ge 2$, we must have $d|a_1x_{1,\ell}$ for $1\le \ell \le j$. In addition, $\mathrm{gcd}(a_1,d)=1$ so we must have $d|x_{1,\ell}$ for $1\le \ell \le j$. So
$$ N'=a_1 \dfrac{x_{1,\ell}}{d} + \sum_{i=2}^k a'_i x_{{i,\ell}},$$
hence $N'$ has at least $j$ distinct positive representations.  But we have already shown that $N'$ cannot have $j+1$ or more positive representations, thus $N'$ has exactly $j$ positive representations.\\ \\
\textbf{Step 3:}  Finally we will show that $N'$ is the largest number with exactly $j$ positive representations by $a_1, a_2', \dots, a_k'$.  Consider any $n>N'$.  Since $dn>dN'=N$, we know that $dn$ can be represented as a linear combination of $a_1,\ldots,a_k$ in exactly $X$ ways with $X\not=j$.\\
Thus, for $1\leq l\leq X$ and $X\not=j$ we have
$$dn=a_1x_{{1,\ell}} + \sum_{i=2}^k a_i x_{{i,\ell}}$$
and as in Step 2,
$$n=a_1(\dfrac{x_{{1,\ell}}}{d}) + \sum_{i=2}^k a'_i x_{{i,\ell}}.$$
If $X>j$ then we certainly do not have exactly $j$ representations, so assume $X<j$. Assume now that we can write $n=a_1y_1 + \sum_{i=2}^k a'_i y_i$
where $y_i\not=x_{i,\ell}$ for any such $\ell$.  By multiplying by $d$ we get a new representation for $dn$, which is a contradiction because $dn$ is represented in exactly $X\not=j$ ways.  
\\\\
Therefore $N'$ is the greatest number with exactly $j$ positive representations and so 
$$N'=f_j(a_1,a'_2,\ldots,a'_k).$$
Thus
$$f_j(a_1,a_2,\ldots,a_k)=d\cdot f_j(a_1,\dfrac{a_2}{d},\ldots,\dfrac{a_k}{d}).$$\end{proof}

Having established our results about $f_j(a_1, \ldots, a_k)$, we show that we can translate these results to results about the $j$-Frobenius Numbers.

\begin{lemma}\label{lemma2}

Either $f_j(a_1, \ldots, a_k) = g_j(a_1, \ldots, a_k)=0$ or, $$f_j(a_1, \ldots, a_k) = g_j(a_1, \ldots, a_k) + \displaystyle \sum_{i=1}^k a_i.$$

\end{lemma}

\begin{proof}For ease, write $f_j$ for $f_j(a_1, \ldots, a_k)$, $g_j$ for $g_j(a_1, \ldots, a_k)$, and $K=\sum_{i=1}^k a_i$.\\ \\
Any representation $(y_1, \ldots, y_k)$ of $M$ gives a representation $(y_1+1, \ldots, y_k +1)$ of $M+K$. Moreover, adding or subtracting $K$ preserves the distinctness of representations because it adjusts every coefficient $y_i$ by $1$. Therefore if $M$ has $j$ representations, $M+K$ has at least $j$ positive representations. Likewise, every positive representation of $M+K$ gives a representation of $M$. Thus $f_j =0$ if and only if $g_j=0$. Assume now that $f_j$ and $g_j$ are both nonzero.\\ \\
Suppose that $f_j < g_j+ K$.  By definition, we can find exactly $j$ representations $(y_1, \ldots, y_k)$ for $g_j$ and $g_j$ has exactly $j$ representations if and only if $g_j + K$ has exactly $j$ \emph{positive} representations $(x_1, \ldots, x_k)$. However, by assumption $g_j + K >f_j$ and $g_j +K$ has exactly $j$ positive representations.  This contradicts the definition of $f_j$, hence $f_j \geq g_j+ K$.\\ \\
Suppose that $f_j > g_j+ K$.  By definition, we can find exactly $j$ positive representations $(x_1, \ldots, x_k)$ for $f_j$.  The same argument as above shows that $f_j - K$ has exactly $j$ representations in contradiction to the definition of $g_j$.  Thus $f_j \leq g_j+ K$.\end{proof}

\textbf{Proof of Theorem \ref{bigresult}:} Combine Theorem \ref{theorem} with Lemma \ref{lemma2}.

\begin{corollary}Let $a_1, a_2$ be coprime positive integers and let $m$ be a positive integer. Suppose that $g_j = g_j(a_1,a_2,ma_1a_2)\ne 0$. Then \begin{itemize}
\item $g_{j} = (j+1)a_1a_2 - a_1-a_2$ for $j< m+1$

\item $g_{m+1} = 0$ and

\item $g_{m+2} = (m+2)a_1a_2-a_1-a_2$.\end{itemize}\end{corollary}

\begin{proof} Theorem \ref{bigresult} tells us that if $g_j(1,1,m) \ne 0$ then \begin{eqnarray*}
g_{j}(a_1, a_2, ma_1a_2) &= & a_2(g_{j}(a_1,1,ma_1)) + (a_2-1)a_1\\
& = & a_2 \left( a_1g_{j}(1,1,m) + (a_1 -1)1\right) + (a_2 -1)a_1\\
& = & a_1a_2 (g_{j}(1,1,m) + 2) -a_1 - a_2 .\\ 
\end{eqnarray*}

Following Beck and Robins in \cite{1}, we can use the values of the restricted partition function $p_{1,1,m}(k)$ to determine $g_{j}(1,1,m)$. Furthermore we can determine the relevant values with the Taylor series $\frac{1}{(1-t)^2(1-t^m)} = \sum_{k=0}^\infty p_{1,1,m}(k)t^k$. Now recall that for $k<m$, $p_{1,1,m}(k) = p_{1,1}(k) = k+1$ but $p_{1,1,m}(m) = m+2$ and for all $k>m$, $p_{1,1,m}(k)>m+2$.  Note that no number is represented $m+1$ times. Thus $g_{m+1}(1,1,m) = 0$, $g_j(1,1,m) = j-1$ for $j<m$ and $g_{m+2}(1,1,m) = m$.\end{proof}

{\bf Remark.} It is a consequence of the asymptotic in Nathanson \cite{3} that for a given tuple, there may be many $j$ for which $g_j =0$, so the ordering $g_0<g_1< \ldots$ may not hold. In the process of discovering these equalities, we noted the somewhat stranger occurrence of tuples where $0< g_{j+1} < g_j$.

Take for instance, the 3-tuple (3, 5, 8). The order $g_0<g_1<\ldots$ holds until $g_{14}=52$ and $g_{15}=51$. As should also be the case, the 3-tuple increased by a factor of $d=2$ creates the new ``dependent'' 3-tuple (3, 10, 16), which fails to hold order in the same position with $g_{14}=107$ and $g_{15}=105$. A few independent examples are as follows:\\

$g_{17}(2, 5, 7) = 43$ and $g_{18}(2, 5, 7) = 42$,

$g_{38}(2, 5, 17)=103$ and $g_{39}(2, 5, 17)=102$,

$g_{35}(4, 7, 19)=181$ and $g_{36}(4, 7, 19)=180$, and

$g_{38}(9, 11, 20)=376$ and $g_{39}(9, 11, 20)=369$.\\

We do not as of yet know a lower bound on $j$ for the above to occur. Indeed, in every case we have computed, if $g_0,g_1>0$ then $g_1>g_0$, but to date neither a proof or a counterexample has presented itself.



\end{document}